\documentclass[11pt]{article}

\usepackage[a4paper,margin=1in]{geometry}
\usepackage{amsmath,amssymb,amsthm,mathtools,mathrsfs}
\usepackage[T1]{fontenc}
\usepackage{lmodern}
\usepackage{microtype}
\usepackage{enumitem}
\usepackage{authblk}
\usepackage{hyperref}

\hypersetup{colorlinks=true,linkcolor=blue,citecolor=blue,urlcolor=blue}

\newtheorem{theorem}{Theorem}[section]
\newtheorem{proposition}[theorem]{Proposition}
\newtheorem{corollary}[theorem]{Corollary}

\theoremstyle{definition}
\newtheorem{definition}[theorem]{Definition}

\newtheorem{remark}[theorem]{Remark}

\newcommand{\Ucal}{\mathcal{U}}
\newcommand{\Vcal}{\mathcal{V}}
\newcommand{\Ctn}{\mathscr{C}}
\newcommand{\Min}{\mathsf{Min}}
\newcommand{\Lin}{\mathsf{Lin}}
\newcommand{\Prod}{\mathsf{Prod}}
\newcommand{\Dr}{\mathsf{Dr}}
\newcommand{\Zero}{\mathsf{Zero}}
\newcommand{\CTNorm}{\mathbf{CTNorm}}
\newcommand{\FiveTypes}{\mathbf{FiveTypes}}
\newcommand{\NN}{\mathbb{N}}

\title{Ultrafilter Equivalence and Asymptotic Types of Five Classical t-Norms}

\author[1]{Jiang Yang\thanks{Corresponding author. Email: \texttt{Jiangyangdy@126.com(yangjiang@nwu.edu.cn)}.}}
\author[2]{Xiongwei Zhang}
\author[1]{Xin Zhang}

\affil[1]{School of Mathematical Sciences, Guangxi Minzu University, Nanning 530006, China}
\affil[2]{Mathematics and Statistics, Yulin University, Yulin, 719000, China}
\date{}

\begin{document}

\maketitle

\begin{abstract}
We study five classical $t$-norms on the unit interval from the viewpoint of ultrafilter concentration. For a fixed ultrafilter $\mathcal U$ on $[0,1]$, we introduce an equivalence relation identifying two operations whenever they coincide on $A\times A$ for some $A\in\mathcal U$.
We show that their asymptotic behavior is governed by two concentration regimes. In the near-$1$ regime, the five operations determine four distinct ultrafilter-equivalence classes. In the low-value regime, the {\L}ukasiewicz, nilpotent minimum, and drastic $t$-norms collapse to the zero operation.
We encode these reductions in a discrete quotient category and record simple ultrametric models for the two regimes. We further interpret the classification inside classical ultrapowers: the near-$1$ and near-$0$ regimes become exact algebraic phenomena on infinitesimal monads, and saturation yields a compactness principle for countable systems of asymptotic identities. Finally, we indicate how the same viewpoint interacts with residual fuzzy implications generated by $t$-norms.
\end{abstract}


\medskip


\section{Introduction}

Triangular norms, or $t$-norms, are basic binary operations in fuzzy logic, many-valued logic, and residuated algebra. Recall that a $t$-norm on $[0,1]$ is a binary operation
\[
T:[0,1]^2\to [0,1]
\]
which is associative, commutative, increasing in each coordinate, and has $1$ as its neutral element.

Among the standard examples, the following five operations occur frequently:
\begin{align*}
T_G(x,y)&=\min\{x,y\}
&&\text{(G\"odel)},\\
T_L(x,y)&=\max\{0,x+y-1\}
&&\text{({\L}ukasiewicz)},\\
T_P(x,y)&=xy
&&\text{(product)},\\
T_{NM}(x,y)&=
\begin{cases}
\min\{x,y\}, & x+y>1,\\
0, & x+y\le 1,
\end{cases}
&&\text{(nilpotent minimum)},\\
T_D(x,y)&=
\begin{cases}
\min\{x,y\}, & \max\{x,y\}=1,\\
0, & \max\{x,y\}<1,
\end{cases}
&&\text{(drastic)}.
\end{align*}
Only the first three are continuous. The nilpotent minimum and the drastic $t$-norm are included here because they are part of the classical list and exhibit important boundary phenomena.

The classical theory of triangular norms is extensive. It includes algebraic and order-theoretic properties, continuity, generators, ordinal sums, comparison and domination, convergence, and applications to many-valued logics and fuzzy set theory; see the monograph of Klement, Mesiar and Pap \cite{KlementMesiarPap}. The present paper does not attempt to add another global classification theorem to this theory. Instead, it asks a more limited question: what remains of these five classical operations when one compares them only on sets that are large with respect to a fixed ultrafilter?

The central idea is to identify two operations when they agree on a set that is large with respect to an ultrafilter. This produces a coarse asymptotic equivalence relation. Such a viewpoint deliberately forgets global distinctions that disappear in a selected endpoint regime, while preserving those distinctions that remain visible asymptotically.

The use of ultrafilters as a language of large sets, generalized limits, and ultraproducts is standard; see Goldbring's monograph \cite{Goldbring}. We use elementary parts of this theory: ultrafilter quantifiers, induced ultrafilters, pushforward ultrafilters, classical ultrapowers, and basic saturation principles. These tools allow us to state the localization, invariance, and nonstandard forms of our equivalence relation cleanly.

There is also a natural connection with the theory of fuzzy implications. Fuzzy implications are one of the central operations in fuzzy logic, and many important families of fuzzy implications are generated from other fuzzy logic operations, especially from $t$-norms via residuation. We refer to the monograph of Baczynski and Jayaram \cite{BaczynskiJayaram} for a systematic account. From this perspective, an asymptotic classification of $t$-norms is relevant not only to fuzzy conjunctions themselves, but also to the endpoint behavior of the fuzzy implications generated by them.

More precisely, we proceed as follows:
\begin{enumerate}[label=\textnormal{(\roman*)}]
\item define an ultrafilter equivalence relation on the family of five classical $t$-norms;
\item analyze the low-value and near-$1$ concentration regimes;
\item prove separation results for the near-$1$ types;
\item record simple ultrametric models for the two regimes;
\item prove a scale-invariance property using pushforward ultrafilters;
\item interpret the classification inside classical ultrapowers and formulate a saturation principle for countable systems of asymptotic identities;
\item explain a residual implication consequence of the near-$1$ classification;
\item encode the resulting identifications in a discrete quotient category.
\end{enumerate}

The low-value regime also clarifies a common informal mistake. The implication
\[
A\subseteq[0,1-\delta]\quad\Longrightarrow\quad T_L=T_{NM}=T_D=0
\]
is false in general. The correct hypothesis is
\[
A\subseteq[0,\delta]\qquad\text{with }\delta\le \frac12.
\]
Near $1$, one must likewise distinguish exact identities from first-order expansions.

\paragraph{Relation with classical structure theory.}
The structure theory of continuous $t$-norms is usually developed through ordinal sums and generators; see, for example, \cite{Hajek,KlementMesiarPap}. From a broader semigroup-theoretic viewpoint, the theorem of Mostert and Shields \cite{MostertShields} is a classical example of a global decomposition theorem for topological semigroups on compact manifolds with boundary.

The present paper should be read in a different, more local-asymptotic sense. We do not use the Mostert--Shields theorem as a technical input for proving a general decomposition theorem for $t$-norms. Rather, our construction is formally analogous to the decomposition philosophy: instead of decomposing an operation globally, we use an ultrafilter to select a large asymptotic regime and then study the induced equivalence class.

Thus the results below are best understood as a finite asymptotic skeleton for the five classical examples, not as a classification theorem for the full class of continuous $t$-norms.

\section{Ultrafilter equivalence}

Let
\[
\Ctn=\{T_G,T_L,T_P,T_{NM},T_D\}
\]
be the family considered in this paper.

\begin{definition}[Ultrafilter equivalence]
Let $\Ucal$ be an ultrafilter on $[0,1]$. For arbitrary binary operations
\[
S,T:[0,1]^2\to[0,1],
\]
we write
\[
S\sim_{\Ucal}T
\]
if there exists $A\in\Ucal$ such that
\[
S(x,y)=T(x,y)\qquad\text{for all }(x,y)\in A\times A.
\]
When no confusion is possible, we restrict this relation to the family $\Ctn$.
\end{definition}

\begin{proposition}
For every ultrafilter $\Ucal$ on $[0,1]$, the relation $\sim_{\Ucal}$ is an equivalence relation on any family of binary operations on $[0,1]$.
\end{proposition}

\begin{proof}
Reflexivity follows by taking $A=[0,1]\in\Ucal$. Symmetry is immediate.

For transitivity, suppose that $S\sim_{\Ucal}T$ and $T\sim_{\Ucal}R$. Then there exist $A,B\in\Ucal$ such that
\[
S=T\quad\text{on }A\times A,
\qquad
T=R\quad\text{on }B\times B.
\]
Since $\Ucal$ is closed under finite intersections, $A\cap B\in\Ucal$. On $(A\cap B)\times(A\cap B)$ we have $S=T=R$. Hence $S\sim_{\Ucal}R$.
\end{proof}

\begin{definition}[Concentration at an endpoint]
Let $\Ucal$ be an ultrafilter on $[0,1]$.
\begin{enumerate}[label=\textnormal{(\roman*)}]
\item We say that $\Ucal$ is concentrated near $1$ if every relative neighbourhood of $1$ belongs to $\Ucal$.
\item We say that $\Ucal$ is non-principal near $1$ if it is concentrated near $1$ and $\{1\}\notin\Ucal$.
\item We say that $\Ucal$ is concentrated near $0$ if every relative neighbourhood of $0$ belongs to $\Ucal$.
\item We say that $\Ucal$ is non-principal near $0$ if it is concentrated near $0$ and $\{0\}\notin\Ucal$.
\end{enumerate}
\end{definition}

\begin{remark}
If $\Ucal$ is non-principal near $1$, then for every $\varepsilon>0$,
\[
(1-\varepsilon,1)\in\Ucal.
\]
Indeed, $(1-\varepsilon,1]\in\Ucal$ by concentration near $1$, while $[0,1)\in\Ucal$ because $\{1\}\notin\Ucal$. Their intersection is $(1-\varepsilon,1)$. Similarly, if $\Ucal$ is non-principal near $0$, then for every $\delta>0$,
\[
(0,\delta)\in\Ucal.
\]
\end{remark}

\begin{definition}[Ultrafilter quantifier]
Let $X$ be a set and let $\Ucal$ be an ultrafilter on $X$. If $\varphi(x)$ is a predicate on $X$, we write
\[
(\Ucal x)\,\varphi(x)
\]
to mean that
\[
\{x\in X:\varphi(x)\}\in\Ucal.
\]
Thus $(\Ucal x)\varphi(x)$ means that $\varphi$ holds for $\Ucal$-almost all $x$.
\end{definition}

\begin{definition}[Induced ultrafilter]
Let $\Ucal$ be an ultrafilter on a set $X$ and let $A\in\Ucal$. The ultrafilter induced by $\Ucal$ on $A$ is
\[
\Ucal|A:=\{B\subseteq A: B=A\cap C\text{ for some }C\in\Ucal\}.
\]
Equivalently, for $B\subseteq A$,
\[
B\in\Ucal|A\quad\Longleftrightarrow\quad B\in\Ucal.
\]
\end{definition}

\begin{remark}
The induced ultrafilter formalizes the localization step used throughout the paper: once a large set $A\in\Ucal$ has been selected, one may work inside $A$ with the induced notion of largeness.
\end{remark}

\begin{definition}[Product ultrafilter]
Let $\Ucal$ be an ultrafilter on a set $X$. The product ultrafilter $\Ucal\otimes\Ucal$ on $X\times X$ is defined by declaring, for $E\subseteq X\times X$,
\[
E\in\Ucal\otimes\Ucal
\]
if and only if
\[
(\Ucal y)(\Ucal x)\bigl((x,y)\in E\bigr).
\]
Equivalently,
\[
E\in\Ucal\otimes\Ucal
\quad\Longleftrightarrow\quad
\{y\in X:\{x\in X:(x,y)\in E\}\in\Ucal\}\in\Ucal.
\]
\end{definition}

\begin{proposition}[Rectangular equivalence implies product-ultrafilter equivalence]
Let $S,T$ be binary operations on $[0,1]$, and put
\[
E_{S,T}:=\{(x,y)\in[0,1]^2:S(x,y)=T(x,y)\}.
\]
If $S\sim_{\Ucal}T$, then
\[
E_{S,T}\in\Ucal\otimes\Ucal.
\]
\end{proposition}

\begin{proof}
Suppose $S\sim_{\Ucal}T$. Then there exists $A\in\Ucal$ such that
\[
A\times A\subseteq E_{S,T}.
\]
For each $y\in A$, we have
\[
A\subseteq \{x\in[0,1]:(x,y)\in E_{S,T}\},
\]
and hence
\[
\{x\in[0,1]:(x,y)\in E_{S,T}\}\in\Ucal.
\]
Therefore
\[
A\subseteq \{y\in[0,1]:\{x\in[0,1]:(x,y)\in E_{S,T}\}\in\Ucal\}.
\]
Since $A\in\Ucal$, the last set belongs to $\Ucal$. Hence $E_{S,T}\in\Ucal\otimes\Ucal$.
\end{proof}

\begin{remark}
We keep the rectangular definition of $\sim_{\Ucal}$ because it gives explicit witnesses $A\in\Ucal$ and is stronger than merely requiring equality on a $\Ucal\otimes\Ucal$-large subset of $[0,1]^2$. The product-ultrafilter formulation is useful for comparison with standard ultrafilter-quantifier notation.
\end{remark}

\section{Two concentration regimes}

\subsection{The low-value regime}

\begin{proposition}[Low-value collapse]
\label{prop:low-collapse}
Let $\Ucal$ be an ultrafilter on $[0,1]$. Assume that there exist $\delta\in(0,\frac12]$ and $A\in\Ucal$ such that
\[
A\subseteq[0,\delta].
\]
Then on $A\times A$ one has
\[
T_L=T_{NM}=T_D=0.
\]
\end{proposition}

\begin{proof}
Take $x,y\in A$. Since $x,y\le\delta\le\frac12$, we have
\[
x+y\le 2\delta\le 1.
\]
For the {\L}ukasiewicz $t$-norm,
\[
T_L(x,y)=\max\{0,x+y-1\}=0.
\]
For the nilpotent minimum, the condition $x+y>1$ never occurs, hence
\[
T_{NM}(x,y)=0.
\]
For the drastic $t$-norm, since $x,y\le\delta<1$, we have
\[
\max\{x,y\}<1,
\]
and therefore
\[
T_D(x,y)=0.
\]
Thus all three operations agree with the zero operation on $A\times A$.
\end{proof}

\begin{remark}
The hypothesis $A\subseteq[0,\delta]$ with $\delta\le\frac12$ is essential. The weaker condition
\[
A\subseteq[0,1-\delta]
\]
does not imply $T_L=0$ or $T_{NM}=0$ on $A\times A$.
\end{remark}

\begin{proposition}[Classes in a non-principal low-value regime]
\label{prop:low-classes}
Let $\Ucal$ be a non-principal ultrafilter concentrated near $0$. Then
\[
[T_L]=[T_{NM}]=[T_D],
\]
and this class is distinct from both $[T_G]$ and $[T_P]$. Moreover,
\[
[T_G]\ne [T_P].
\]
Thus the five operations determine exactly three $\Ucal$-classes:
\[
[T_L]=[T_{NM}]=[T_D],\qquad [T_G],\qquad [T_P].
\]
\end{proposition}

\begin{proof}
Since $\Ucal$ is concentrated near $0$, the set $(0,\delta)\in\Ucal$ for every $\delta>0$. Taking $\delta\le\frac12$ and applying Proposition~\ref{prop:low-collapse}, we get
\[
[T_L]=[T_{NM}]=[T_D].
\]

We now prove that this zero class is distinct from $[T_G]$ and $[T_P]$. Suppose, for example, that $T_G\sim_{\Ucal}T_L$. Then there exists $B\in\Ucal$ such that
\[
T_G(x,y)=0\qquad\text{for all }(x,y)\in B\times B.
\]
Intersecting $B$ with $(0,\delta)$, we may assume $B\subseteq(0,\delta)$ and $B\in\Ucal$. Taking $x=y=t\in B$, we get $t=0$, which contradicts $B\subseteq(0,\delta)$. Hence $[T_G]$ is distinct from the zero class.

Similarly, if $T_P\sim_{\Ucal}T_L$, then on some $B\in\Ucal$ with $B\subseteq(0,\delta)$ we would have $t^2=0$ for every $t\in B$, again impossible. Hence $[T_P]$ is distinct from the zero class.

Finally, suppose $T_G\sim_{\Ucal}T_P$. Then on some $B\in\Ucal$, after intersecting with $(0,\delta)$, we have $t=t^2$ for every $t\in B$. Thus $t\in\{0,1\}$ for all $t\in B$, impossible because $B\subseteq(0,\delta)$ and $\delta<1$ while $0\notin B$. Therefore $[T_G]\ne [T_P]$.
\end{proof}

\subsection{The near-1 regime}

\begin{proposition}[Exact formulas near $1$]
\label{prop:near-1-expansion}
Let $\varepsilon\in(0,\frac12)$ and let
\[
A\subseteq[1-\varepsilon,1].
\]
Then, for all $(x,y)\in A\times A$,
\begin{align*}
T_G(x,y)&=\min\{x,y\},\\
T_L(x,y)&=x+y-1,\\
T_P(x,y)&=1-[(1-x)+(1-y)]+(1-x)(1-y),\\
T_{NM}(x,y)&=\min\{x,y\},\\
T_D(x,y)&=
\begin{cases}
\min\{x,y\}, & \max\{x,y\}=1,\\
0, & \max\{x,y\}<1.
\end{cases}
\end{align*}
Moreover,
\[
T_P(x,y)=1-[(1-x)+(1-y)] + o\bigl(\|(1-x,1-y)\|\bigr)
\]
as $(x,y)\to(1,1)$.
\end{proposition}

\begin{proof}
The formulas for $T_G$ and $T_D$ are their definitions.

For $T_L$, since $x,y\in[1-\varepsilon,1]$ and $\varepsilon<\frac12$, we have
\[
x+y\ge 2(1-\varepsilon)>1.
\]
Hence the positive branch is active:
\[
T_L(x,y)=\max\{0,x+y-1\}=x+y-1.
\]

For $T_{NM}$, the same inequality $x+y>1$ implies
\[
T_{NM}(x,y)=\min\{x,y\}.
\]

For $T_P$, write $x=1-u$ and $y=1-v$, where $u=1-x$ and $v=1-y$. Then
\[
xy=(1-u)(1-v)=1-u-v+uv,
\]
or equivalently
\[
T_P(x,y)=1-[(1-x)+(1-y)]+(1-x)(1-y).
\]
Finally,
\[
|uv|\le \frac12(u^2+v^2)\le \|(u,v)\|^2.
\]
Consequently,
\[
\frac{|uv|}{\|(u,v)\|}\le \|(u,v)\|\to 0
\qquad ((u,v)\to(0,0)).
\]
Thus $uv=o(\|(u,v)\|)$, which gives the asserted first-order expansion.
\end{proof}

\begin{remark}
Near $1$, the five classical operations exhibit four labelled behaviours:
\[
\Min,\qquad \Lin,\qquad \Prod,\qquad \Dr.
\]
The label $\Min$ is represented by both $T_G$ and $T_{NM}$. The label $\Dr$ denotes the $\Ucal$-class represented by $T_D$ in the near-$1$ quotient; on a non-principal near-$1$ large set, its local formula is the zero operation.
\end{remark}

\section{Separation of the near-1 types}

\begin{proposition}[Pairwise inequivalence near $1$]
\label{prop:pairwise}
Let $\Ucal$ be a non-principal ultrafilter concentrated near $1$. Then the $\Ucal$-equivalence classes
\[
[T_G]=[T_{NM}],\qquad [T_L],\qquad [T_P],\qquad [T_D]
\]
are pairwise distinct.
\end{proposition}

\begin{proof}
Since $\Ucal$ is non-principal near $1$, for every $\varepsilon\in(0,\frac12)$ the set $A_\varepsilon=(1-\varepsilon,1)$ belongs to $\Ucal$.

First, Proposition~\ref{prop:near-1-expansion} gives $T_G=T_{NM}$ on $[1-\varepsilon,1]\times[1-\varepsilon,1]$, hence $[T_G]=[T_{NM}]$.

We now show that no further identifications occur.

\smallskip
\noindent\emph{Step 1: $[T_G]\ne[T_L]$.}
Suppose toward a contradiction that $T_G\sim_{\Ucal}T_L$. Then there exists $B\in\Ucal$ such that
\[
\min\{x,y\}=x+y-1
\qquad\text{for all }(x,y)\in B\times B.
\]
Let $C=B\cap A_\varepsilon$. Then $C\in\Ucal$ and $C\subseteq(1-\varepsilon,1)$. Taking $x=y=t\in C$, we get $t=2t-1$, so $t=1$. This contradicts $C\subseteq(1-\varepsilon,1)$. Hence $[T_G]\ne[T_L]$.

\smallskip
\noindent\emph{Step 2: $[T_G]\ne[T_P]$.}
Suppose $T_G\sim_{\Ucal}T_P$. Then for some $B\in\Ucal$,
\[
\min\{x,y\}=xy
\qquad\text{for all }(x,y)\in B\times B.
\]
Intersecting again with $A_\varepsilon$, take $t\in C=B\cap A_\varepsilon$. Then $t=t^2$, so $t\in\{0,1\}$, contradicting $C\subseteq(1-\varepsilon,1)$. Thus $[T_G]\ne[T_P]$.

\smallskip
\noindent\emph{Step 3: $[T_L]\ne[T_P]$.}
Suppose $T_L\sim_{\Ucal}T_P$. Then for some $B\in\Ucal$,
\[
x+y-1=xy
\qquad\text{for all }(x,y)\in B\times B.
\]
Taking $x=y=t$ in $C=B\cap A_\varepsilon$, we obtain $2t-1=t^2$, or $(1-t)^2=0$. Thus $t=1$, again contradicting $C\subseteq(1-\varepsilon,1)$. Hence $[T_L]\ne[T_P]$.

\smallskip
\noindent\emph{Step 4: the class represented by $T_D$ is distinct.}
On $A_\varepsilon\times A_\varepsilon$ we have $T_D(x,y)=0$ because $\max\{x,y\}<1$. By contrast,
\[
T_G(x,y)>0,\qquad T_L(x,y)>0,\qquad T_P(x,y)>0
\]
on $A_\varepsilon\times A_\varepsilon$. Therefore $T_D$ cannot be $\Ucal$-equivalent to any of $T_G,T_L,T_P$.
\end{proof}

\begin{corollary}[Four classes in the near-$1$ regime]
\label{cor:fourclasses}
If $\Ucal$ is a non-principal ultrafilter concentrated near $1$, then the five classical operations determine exactly four $\Ucal$-equivalence classes:
\[
[T_G]=[T_{NM}],\qquad [T_L],\qquad [T_P],\qquad [T_D].
\]
\end{corollary}

\begin{proof}
This is exactly Proposition~\ref{prop:pairwise}.
\end{proof}

\section{Ultrametric models}

Although the quotient classification above does not require metric machinery, it is useful to record simple ultrametric models for the two regimes.

\begin{proposition}[Ultrametric for the near-$1$ regime]
\label{prop:ultra-near1}
Let $A\subseteq[1-\varepsilon,1]$ with $\varepsilon\in(0,\frac12)$. Define
\[
d_1(x,y)=
\begin{cases}
0, & x=y,\\
\max\{1-x,1-y\}, & x\ne y.
\end{cases}
\]
Then $d_1$ is an ultrametric on $A$.
\end{proposition}

\begin{proof}
Symmetry and definiteness are immediate. Let $x,y,z\in A$. If $x=z$, then $d_1(x,z)=0\le \max\{d_1(x,y),d_1(y,z)\}$. If $x\ne z$, then
\[
d_1(x,z)=\max\{1-x,1-z\}.
\]
On the other hand,
\begin{align*}
\max\{d_1(x,y),d_1(y,z)\}
&=\max\bigl\{\max\{1-x,1-y\},\max\{1-y,1-z\}\bigr\}\\
&=\max\{1-x,1-y,1-z\}.
\end{align*}
Thus $d_1(x,z)\le \max\{d_1(x,y),d_1(y,z)\}$. Hence $d_1$ is an ultrametric.
\end{proof}

\begin{proposition}[Ultrametric for the low-value regime]
\label{prop:ultra-low}
Let $A\subseteq[0,\delta]$ with $\delta\in(0,\frac12]$. Define
\[
d_0(x,y)=
\begin{cases}
0, & x=y,\\
\max\{x,y\}, & x\ne y.
\end{cases}
\]
Then $d_0$ is an ultrametric on $A$.
\end{proposition}

\begin{proof}
Symmetry and definiteness are immediate. Let $x,y,z\in A$. If $x=z$, the ultrametric inequality is immediate. Otherwise,
\[
d_0(x,z)=\max\{x,z\}.
\]
Moreover,
\begin{align*}
\max\{d_0(x,y),d_0(y,z)\}
&=\max\bigl\{\max\{x,y\},\max\{y,z\}\bigr\}\\
&=\max\{x,y,z\}.
\end{align*}
Therefore $d_0(x,z)\le \max\{d_0(x,y),d_0(y,z)\}$, and $d_0$ is an ultrametric.
\end{proof}

\begin{remark}
The ultrametrics $d_1$ and $d_0$ are used here only as regime-specific geometric models. We do not claim that they recover the Euclidean topology on the whole interval, nor do we claim that they yield a complete metric classification of all $t$-norms.
\end{remark}

\section{Pushforward ultrafilters and scale invariance}

The category-theoretic treatment of ultrafilters naturally involves pushforward ultrafilters. We record here a simple invariance property of the equivalence relation $\sim_{\Ucal}$.

\begin{definition}[Pushforward ultrafilter]
Let $X,Y$ be sets, let $f:X\to Y$ be a map, and let $\Ucal$ be an ultrafilter on $X$. The pushforward ultrafilter $f_*\Ucal$ on $Y$ is defined by
\[
B\in f_*\Ucal
\quad\Longleftrightarrow\quad
f^{-1}(B)\in\Ucal
\]
for every $B\subseteq Y$.
\end{definition}

\begin{definition}[Change of scale]
Let $h:[0,1]\to[0,1]$ be a strictly increasing homeomorphism with $h(0)=0$ and $h(1)=1$. For a $t$-norm $T$ on $[0,1]$, define its $h$-conjugate by
\[
T^h(a,b):=h\bigl(T(h^{-1}(a),h^{-1}(b))\bigr).
\]
\end{definition}

\begin{proposition}[Scale invariance of ultrafilter equivalence]
\label{prop:scale-invariance}
Let $h:[0,1]\to[0,1]$ be a strictly increasing homeomorphism fixing $0$ and $1$. If $S,T:[0,1]^2\to[0,1]$ satisfy $S\sim_{\Ucal}T$, then
\[
S^h\sim_{h_*\Ucal}T^h.
\]
\end{proposition}

\begin{proof}
Since $S\sim_{\Ucal}T$, there exists $A\in\Ucal$ such that $S(x,y)=T(x,y)$ for all $(x,y)\in A\times A$. Let $B=h(A)$. Then $B\in h_*\Ucal$, because $h^{-1}(B)=A\in\Ucal$.

Now take $a,b\in B$. Then $a=h(x)$ and $b=h(y)$ for some $x,y\in A$. Hence
\[
S^h(a,b)=h(S(x,y))=h(T(x,y))=T^h(a,b).
\]
Thus $S^h=T^h$ on $B\times B$ for some $B\in h_*\Ucal$, and therefore $S^h\sim_{h_*\Ucal}T^h$.
\end{proof}

\begin{remark}
Proposition~\ref{prop:scale-invariance} shows that the quotient construction does not depend on a particular choice of numerical scale on $[0,1]$. An order-homeomorphic change of scale transports the ultrafilter and preserves ultrafilter equivalence.
\end{remark}

\section{Ultrapower interpretation and saturation}

We now reinterpret the preceding asymptotic calculations inside classical ultrapowers. This section is not needed for the elementary quotient classification, but it explains why the two concentration regimes are naturally viewed as infinitesimal monads. It also gives a saturation principle for countable systems of asymptotic identities.

Throughout this section, let $\Vcal$ be a non-principal ultrafilter on $\NN$. Let
\[
{}^{\ast}\mathbb R=\mathbb R^{\mathbb N}/\Vcal
\]
be the corresponding ultrapower of the ordered field of real numbers, and let ${}^{\ast}[0,1]$ be the induced ultrapower of the interval $[0,1]$.

For a function $T:[0,1]^2\to[0,1]$, we write
\[
{}^{\ast}T:{}^{\ast}[0,1]^2\to{}^{\ast}[0,1]
\]
for its coordinatewise ultrapower extension:
\[
{}^{\ast}T\bigl([x_n]_{\Vcal},[y_n]_{\Vcal}\bigr)
=
[T(x_n,y_n)]_{\Vcal}.
\]

\begin{definition}[Infinitesimal endpoint monads]
Define the left monad of $1$ in ${}^{\ast}[0,1]$ by
\[
\mu^-(1)=
\left\{
\xi\in{}^{\ast}[0,1]:
0<1-\xi<\frac1n\text{ for every }n\ge2
\right\}.
\]
Define the right monad of $0$ in ${}^{\ast}[0,1]$ by
\[
\mu^+(0)=
\left\{
\xi\in{}^{\ast}[0,1]:
0<\xi<\frac1n\text{ for every }n\ge2
\right\}.
\]
\end{definition}

\begin{remark}
Elements of $\mu^-(1)$ are positive infinitesimal perturbations from the left of $1$, while elements of $\mu^+(0)$ are positive infinitesimals above $0$. For example, in any non-principal ultrapower over $\NN$, the classes of the sequences $1-1/(n+2)$ and $1/(n+2)$ belong to $\mu^-(1)$ and $\mu^+(0)$, respectively.
\end{remark}

\begin{proposition}[Nonstandard near-$1$ classification]
\label{prop:nonstandard-near1}
Let $\xi,\eta\in\mu^-(1)$. Then
\begin{align*}
{}^{\ast}T_G(\xi,\eta)&={}^{\ast}T_{NM}(\xi,\eta)=\min\{\xi,\eta\},\\
{}^{\ast}T_L(\xi,\eta)&=\xi+\eta-1,\\
{}^{\ast}T_P(\xi,\eta)&=\xi\eta,\\
{}^{\ast}T_D(\xi,\eta)&=0.
\end{align*}
\end{proposition}

\begin{proof}
Since $\xi,\eta\in\mu^-(1)$, we have
\[
\xi>1-\frac13=\frac23,
\qquad
\eta>1-\frac13=\frac23.
\]
Hence $\xi+\eta>1$. Therefore the positive branch of the {\L}ukasiewicz operation is active, and
\[
{}^{\ast}T_L(\xi,\eta)=\xi+\eta-1.
\]
For the nilpotent minimum, the same inequality $\xi+\eta>1$ gives
\[
{}^{\ast}T_{NM}(\xi,\eta)=\min\{\xi,\eta\}.
\]
The formula for $T_G$ is immediate from its definition, and the formula for $T_P$ is coordinatewise multiplication. Finally, since $\xi<1$ and $\eta<1$, we have $\max\{\xi,\eta\}<1$, so the drastic operation gives ${}^{\ast}T_D(\xi,\eta)=0$.
\end{proof}

\begin{proposition}[Nonstandard low-value collapse]
\label{prop:nonstandard-low}
Let $\xi,\eta\in\mu^+(0)$. Then
\[
{}^{\ast}T_L(\xi,\eta)
=
{}^{\ast}T_{NM}(\xi,\eta)
=
{}^{\ast}T_D(\xi,\eta)
=0.
\]
\end{proposition}

\begin{proof}
Since $\xi,\eta\in\mu^+(0)$, we have $0<\xi<1/3$ and $0<\eta<1/3$. Thus $\xi+\eta<2/3<1$. Consequently,
\[
{}^{\ast}T_L(\xi,\eta)=\max\{0,\xi+\eta-1\}=0,
\]
and the condition $\xi+\eta>1$ in the nilpotent minimum never occurs, so ${}^{\ast}T_{NM}(\xi,\eta)=0$. Moreover, since $\xi,\eta<1$, we have $\max\{\xi,\eta\}<1$, and hence ${}^{\ast}T_D(\xi,\eta)=0$.
\end{proof}

\begin{proposition}[Saturated separation near $1$]
\label{prop:saturated-separation}
Let $\xi\in\mu^-(1)$. Then
\[
{}^{\ast}T_G(\xi,\xi)=\xi,
\qquad
{}^{\ast}T_L(\xi,\xi)=2\xi-1,
\qquad
{}^{\ast}T_P(\xi,\xi)=\xi^2,
\qquad
{}^{\ast}T_D(\xi,\xi)=0,
\]
and these four hyperreal numbers are pairwise distinct.
\end{proposition}

\begin{proof}
Put $\epsilon=1-\xi$. Since $\xi\in\mu^-(1)$, we have $0<\epsilon<1/n$ for every $n\ge2$. In particular, $0<\epsilon<1$ and $\xi=1-\epsilon>0$.

By Proposition~\ref{prop:nonstandard-near1}, the four displayed formulas hold. We now compare the values. First,
\[
\xi-(2\xi-1)=1-\xi=\epsilon>0.
\]
Second,
\[
\xi-\xi^2=\xi(1-\xi)=\xi\epsilon>0.
\]
Third,
\[
\xi^2-(2\xi-1)=(1-\xi)^2=\epsilon^2>0.
\]
Since $\xi>0$, $\xi^2>0$, and $2\xi-1>0$, none of these three values is equal to $0$. Hence all four displayed values are pairwise distinct.
\end{proof}

\begin{theorem}[Local equality versus monad equality]
\label{thm:local-vs-monad}
Let $S,T:[0,1]^2\to[0,1]$ be arbitrary binary operations.
\begin{enumerate}[label=\textnormal{(\arabic*)}]
\item There exists $\varepsilon>0$ such that $S=T$ on $(1-\varepsilon,1)^2$ if and only if ${}^{\ast}S={}^{\ast}T$ on $\mu^-(1)\times\mu^-(1)$.
\item There exists $\delta>0$ such that $S=T$ on $(0,\delta)^2$ if and only if ${}^{\ast}S={}^{\ast}T$ on $\mu^+(0)\times\mu^+(0)$.
\end{enumerate}
\end{theorem}

\begin{proof}
We prove the near-$1$ statement; the low-value statement is analogous.

Assume first that there exists $\varepsilon>0$ such that $S=T$ on $(1-\varepsilon,1)^2$. Let $\xi=[x_n]_{\Vcal}$ and $\eta=[y_n]_{\Vcal}$ belong to $\mu^-(1)$. Choose $k\ge2$ such that $1/k<\varepsilon$. Since $\xi,\eta\in\mu^-(1)$, the sets
\[
\{n:x_n\in(1-\varepsilon,1)\},
\qquad
\{n:y_n\in(1-\varepsilon,1)\}
\]
belong to $\Vcal$. On their intersection, also in $\Vcal$, we have $S(x_n,y_n)=T(x_n,y_n)$. Therefore ${}^{\ast}S(\xi,\eta)={}^{\ast}T(\xi,\eta)$.

Conversely, suppose that no such $\varepsilon$ exists. Then for every $m\ge2$ there exist $x_m,y_m\in(1-1/m,1)$ such that $S(x_m,y_m)\ne T(x_m,y_m)$. Define $\xi=[x_m]_{\Vcal}$ and $\eta=[y_m]_{\Vcal}$. For each fixed $k\ge2$, the set
\[
\{m:x_m,y_m\in(1-1/k,1)\}
\]
contains all sufficiently large $m$, and hence belongs to the non-principal ultrafilter $\Vcal$. Thus $\xi,\eta\in\mu^-(1)$. But $S(x_m,y_m)\ne T(x_m,y_m)$ for every $m\ge2$, so ${}^{\ast}S(\xi,\eta)\ne{}^{\ast}T(\xi,\eta)$. This contradicts monad equality. Hence standard local equality must hold on some interval $(1-\varepsilon,1)^2$.
\end{proof}

\begin{remark}
Theorem~\ref{thm:local-vs-monad} explains why the ultrapower viewpoint is not merely a change of language. It shows that equality on an infinitesimal monad is equivalent to equality on a sufficiently small standard endpoint neighbourhood.
\end{remark}

\begin{theorem}[Saturated realization of asymptotic regimes]
\label{thm:saturated-realization}
Assume that the ultrapower ${}^{\ast}\mathbb R$ is $\aleph_1$-saturated. Then both types
\[
p_1(x)=\left\{1-\frac1n<x<1:n\ge2\right\}
\]
and
\[
p_0(x)=\left\{0<x<\frac1n:n\ge2\right\}
\]
are realized in ${}^{\ast}[0,1]$. Equivalently,
\[
\mu^-(1)\ne\varnothing
\qquad\text{and}\qquad
\mu^+(0)\ne\varnothing.
\]
\end{theorem}

\begin{proof}
Every finite subset of $p_1(x)$ is satisfiable in $[0,1]$: if the finite subset involves only $1-1/n<x<1$ for $2\le n\le N$, choose any real number $x\in(1-1/N,1)$. Thus $p_1(x)$ is finitely satisfiable. By $\aleph_1$-saturation, it is realized in ${}^{\ast}[0,1]$.

The proof for $p_0(x)$ is identical. A finite subset only requires $0<x<1/n$ for finitely many $n$, say $2\le n\le N$. Choosing $x\in(0,1/N)$ realizes that finite subset. Saturation gives a realization of the full type.
\end{proof}

\begin{theorem}[Saturation principle for asymptotic identities]
\label{thm:saturation-principle}
Let $L$ be a countable first-order language expanding the language of ordered sets by constants for all rational numbers in $[0,1]$ and by function symbols for the five operations in $\Ctn$. Let $M$ be the corresponding structure with universe $[0,1]$, and let ${}^{\ast}M$ be an $\aleph_1$-saturated ultrapower.

Let $\Sigma(\bar x)$ be a countable set of $L$-formulas in variables $\bar x=(x_1,\dots,x_m)$. Assume that $\Sigma$ is finitely satisfiable arbitrarily close to $1$, in the following sense: for every finite $\Sigma_0\subseteq\Sigma$ and every $k\ge2$, there exists $\bar a\in(1-1/k,1)^m$ such that
\[
M\models \bigwedge_{\varphi\in\Sigma_0}\varphi(\bar a).
\]
Then there exists $\bar \xi=(\xi_1,\dots,\xi_m)\in\mu^-(1)^m$ such that
\[
{}^{\ast}M\models \varphi(\bar \xi)
\]
for every $\varphi\in\Sigma$. The analogous statement holds near $0$ with $\mu^+(0)$ in place of $\mu^-(1)$.
\end{theorem}

\begin{proof}
Consider the countable type
\[
q(\bar x)=
\Sigma(\bar x)\cup
\left\{1-\frac1n<x_i<1:n\ge2,\;1\le i\le m\right\}.
\]
We show that every finite subset of $q$ is satisfiable in $M$.

Let $q_0\subseteq q$ be finite. Then $q_0$ contains only finitely many formulas from $\Sigma$ and finitely many inequalities of the form $1-1/n<x_i<1$. Let $\Sigma_0$ be the finite subset of $\Sigma$ appearing in $q_0$, and choose $k\ge2$ large enough so that all endpoint inequalities in $q_0$ are implied by
\[
1-\frac1k<x_i<1
\qquad (1\le i\le m).
\]
By the finite satisfiability assumption, there exists $\bar a\in(1-1/k,1)^m$ realizing all formulas in $\Sigma_0$. Therefore $\bar a$ realizes $q_0$.

Thus $q$ is finitely satisfiable. Since ${}^{\ast}M$ is $\aleph_1$-saturated, $q$ is realized by some tuple $\bar\xi\in{}^{\ast}M^m$. The endpoint inequalities in $q$ ensure that $\bar\xi\in\mu^-(1)^m$, and the formulas from $\Sigma$ are all realized by $\bar\xi$.

The proof near $0$ is obtained by replacing the inequalities $1-1/n<x_i<1$ with $0<x_i<1/n$.
\end{proof}

\begin{remark}
Theorem~\ref{thm:saturation-principle} is the model-theoretic form of the asymptotic arguments used earlier. Instead of choosing a new sufficiently small neighbourhood for each finite calculation, saturation realizes all countably many compatible asymptotic requirements at once inside an ultrapower.
\end{remark}

\begin{remark}
We deliberately use classical ultrapowers rather than metric ultraproducts. The operations $T_{NM}$ and $T_D$ are not continuous, so treating all five operations as function symbols in continuous logic would require additional care. The classical ultrapower framework avoids this issue and is sufficient for the infinitesimal classifications proved above.
\end{remark}

\section{Residual implications and endpoint asymptotics}

We now record a consequence for fuzzy implications. This section is not intended to develop a full asymptotic theory of fuzzy implications. Its purpose is to show that the ultrafilter classification of $t$-norms naturally interacts with residual implications, which form one of the central classes of fuzzy implications studied in \cite{BaczynskiJayaram}.

\begin{definition}[Residual implication of a t-norm]
Let $T$ be a $t$-norm on $[0,1]$. Its residual implication is the operation
\[
I_T:[0,1]^2\to[0,1]
\]
defined by
\[
I_T(x,y)=\sup\{z\in[0,1]:T(x,z)\le y\}.
\]
\end{definition}

\begin{remark}
Residual implications are often called R-implications in the fuzzy implication literature. They are generated from fuzzy conjunctions by residuation and provide a standard bridge between $t$-norms and implication operations.
\end{remark}

\begin{proposition}[A residual implication consequence]
\label{prop:residual-consequence}
Let $\varepsilon\in(0,\frac12)$ and let
\[
A\subseteq[1-\varepsilon,1].
\]
Let $I_G$ and $I_{NM}$ be the residual implications associated with $T_G$ and $T_{NM}$, respectively. Then
\[
I_G(x,y)=I_{NM}(x,y)
\]
for all $(x,y)\in A\times A$. Consequently, if $A\in\Ucal$, then
\[
I_G\sim_{\Ucal} I_{NM}.
\]
\end{proposition}

\begin{proof}
Recall that the residual implication of the G\"odel $t$-norm is
\[
I_G(x,y)=
\begin{cases}
1, & x\le y,\\
y, & x>y.
\end{cases}
\]
For the nilpotent minimum $t$-norm, the associated residual implication is
\[
I_{NM}(x,y)=
\begin{cases}
1, & x\le y,\\
\max\{1-x,y\}, & x>y.
\end{cases}
\]
Let $x,y\in A$. If $x\le y$, then both implications are equal to $1$.

Assume now that $x>y$. Since $x,y\in[1-\varepsilon,1]$, we have $y\ge 1-\varepsilon$. Also, $1-x\le \varepsilon$. Because $\varepsilon<1/2$, we have $1-\varepsilon>\varepsilon$. Thus
\[
y\ge 1-\varepsilon>\varepsilon\ge 1-x.
\]
Consequently, $\max\{1-x,y\}=y$. Therefore
\[
I_{NM}(x,y)=y=I_G(x,y).
\]
Hence $I_G=I_{NM}$ on $A\times A$, and the ultrafilter statement follows immediately from the definition of $\sim_{\Ucal}$.
\end{proof}

\begin{remark}[Meaning for fuzzy implications]
Proposition~\ref{prop:residual-consequence} shows that ultrafilter equivalence of $t$-norms can pass to at least some naturally associated fuzzy implications. In particular, the near-$1$ identification
\[
T_G\sim_{\Ucal}T_{NM}
\]
has a corresponding implication-level manifestation:
\[
I_G\sim_{\Ucal}I_{NM}.
\]
Thus the endpoint asymptotic skeleton obtained for $t$-norms may be useful for studying coarse asymptotic behavior of residual fuzzy implications.
\end{remark}

\begin{remark}
The result above is deliberately restricted to a concrete and transparent case. A full extension of the present quotient theory from $t$-norms to residual, $(S,N)$-, QL-, and other families of fuzzy implications would require separate analysis. The point is that the present paper supplies a possible organizing principle: one may first identify asymptotic types of fuzzy conjunctions and then ask which implication operations preserve, refine, or destroy these types.
\end{remark}

\section{The quotient category}

\begin{definition}[The quotient category]
Let $\CTNorm_{\Ucal}$ be the category defined as follows:
\begin{itemize}[leftmargin=2.2em]
\item the objects are the $\sim_{\Ucal}$-equivalence classes $[T]$ with $T\in\Ctn$;
\item for two objects $[S]$ and $[T]$,
\[
\operatorname{Hom}_{\CTNorm_{\Ucal}}([S],[T])=
\begin{cases}
\{\mathrm{id}_{[S]}\}, & [S]=[T],\\
\varnothing, & [S]\ne [T].
\end{cases}
\]
\end{itemize}
\end{definition}

\begin{remark}
Thus $\CTNorm_{\Ucal}$ is a discrete category. Its mathematical content is the quotient set of $\Ctn$ under $\sim_{\Ucal}$.
\end{remark}

\begin{theorem}[Main classification theorem]
\label{thm:main}
Let $\Ucal$ be an ultrafilter on $[0,1]$.
\begin{enumerate}[label=\textnormal{(\arabic*)}]
\item If $\Ucal$ is non-principal and concentrated near $1$, then
\[
\CTNorm_{\Ucal}\cong \{\,\Min,\Lin,\Prod,\Dr\,\}
\]
as a discrete category. Here $\Min$ is represented by $T_G$ and $T_{NM}$, while $\Lin$, $\Prod$, and $\Dr$ are represented by $T_L$, $T_P$, and $T_D$, respectively.
\item If $\Ucal$ is non-principal and concentrated near $0$, then
\[
\CTNorm_{\Ucal}\cong \{\,\Zero,\Min,\Prod\,\}
\]
as a discrete category. Here $\Zero$ is represented by $T_L,T_{NM},T_D$, while $\Min$ and $\Prod$ are represented by $T_G$ and $T_P$, respectively.
\end{enumerate}
In particular, for every ultrafilter $\Ucal$, the quotient category $\CTNorm_{\Ucal}$ is finite and has at most five objects.
\end{theorem}

\begin{proof}
If $\Ucal$ is non-principal and concentrated near $1$, the statement follows from Corollary~\ref{cor:fourclasses}. The labels $\Min$, $\Lin$, $\Prod$, and $\Dr$ are simply names for the four quotient classes.

If $\Ucal$ is non-principal and concentrated near $0$, the statement follows from Proposition~\ref{prop:low-classes}.

For an arbitrary ultrafilter $\Ucal$, the object set of $\CTNorm_{\Ucal}$ is a quotient of the five-element set
\[
\Ctn=\{T_G,T_L,T_P,T_{NM},T_D\}.
\]
Hence it has at most five elements, and the associated discrete category is finite.
\end{proof}

\section{Comparison with pointwise equality and ordinal-sum theory}

We next compare ultrafilter equivalence with finer equivalence relations.

\begin{proposition}
\label{prop:pointwise-coarser}
Ultrafilter equivalence is strictly coarser than pointwise equality on $\Ctn$.
\end{proposition}

\begin{proof}
If two operations are pointwise equal on $[0,1]^2$, then they are $\Ucal$-equivalent for every ultrafilter $\Ucal$, by taking $A=[0,1]$.

The converse fails. Let $\Ucal$ be a non-principal ultrafilter concentrated near $1$. Then $[T_G]=[T_{NM}]$ by Proposition~\ref{prop:pairwise}. However, $T_G$ and $T_{NM}$ are not pointwise equal on $[0,1]^2$. For example,
\[
T_G\left(\frac14,\frac14\right)=\frac14,
\]
whereas
\[
T_{NM}\left(\frac14,\frac14\right)=0.
\]
Thus ultrafilter equivalence is strictly coarser than pointwise equality.
\end{proof}

\begin{remark}[Relation with ordinal sums]
The classical structure theory of continuous $t$-norms is based on ordinal sums and additive generators. Since the present paper includes non-continuous examples such as $T_{NM}$ and $T_D$, and since we restrict attention to the five operations in $\Ctn$, we do not formulate a general theorem comparing $\sim_{\Ucal}$ with ordinal-sum equivalence for all continuous $t$-norms.

Nevertheless, Proposition~\ref{prop:pointwise-coarser} illustrates the intended coarse-graining phenomenon: ultrafilter equivalence forgets global distinctions that disappear on a large asymptotic set.
\end{remark}

\begin{remark}[Relation with Mostert--Shields]
Mostert and Shields \cite{MostertShields} studied the structure of semigroups on compact manifolds with boundary. Their theorem is a global structural result for topological semigroups. Our approach is different: it extracts a large asymptotic regime by means of an ultrafilter and then compares the induced restrictions of finitely many operations.

Thus the present classification should not be viewed as a direct consequence of the Mostert--Shields theorem. Rather, it is an ultrafilter-based analogue of the general idea that global structure may be simplified by passing to local or asymptotic components.
\end{remark}

\section{Functorial interpretation}

\begin{definition}
Let $\CTNorm$ be the discrete category whose objects are the five operations
\[
T_G,
\quad
T_L,
\quad
T_P,
\quad
T_{NM},
\quad
T_D.
\]
Let $\FiveTypes$ be the discrete category whose objects are the formal symbols
\[
\Min,
\quad
\Lin,
\quad
\Prod,
\quad
\Dr,
\quad
\Zero.
\]
\end{definition}

\begin{proposition}[Near-$1$ reduction functor]
Let $\Ucal$ be a non-principal ultrafilter concentrated near $1$. There is a functor
\[
F_1:\CTNorm\longrightarrow \FiveTypes
\]
defined on objects by
\[
F_1(T)=
\begin{cases}
\Min, & T=T_G\text{ or }T=T_{NM},\\
\Lin, & T=T_L,\\
\Prod, & T=T_P,\\
\Dr, & T=T_D.
\end{cases}
\]
It factors through the quotient category $\CTNorm_{\Ucal}$. It is not injective on objects, but it is faithful.
\end{proposition}

\begin{proof}
Since $\CTNorm$ and $\FiveTypes$ are discrete categories, a map on objects determines a unique functor by sending each identity morphism to the identity morphism of its image.

If $S\sim_{\Ucal}T$, then $S$ and $T$ represent the same quotient class. By the near-$1$ classification, this implies $F_1(S)=F_1(T)$. Hence $F_1$ factors through $\CTNorm_{\Ucal}$.

The functor is not injective on objects because $F_1(T_G)=F_1(T_{NM})=\Min$ while $T_G\ne T_{NM}$. On the other hand, since the source and target categories are discrete, every hom-set is either empty or a singleton. Therefore the induced map on every hom-set is injective, and $F_1$ is faithful.
\end{proof}

\begin{proposition}[Low-value reduction functor]
Let $\Ucal$ be a non-principal ultrafilter concentrated near $0$. There is a functor
\[
F_0:\CTNorm\longrightarrow \FiveTypes
\]
defined on objects by
\[
F_0(T)=
\begin{cases}
\Zero, & T=T_L,\; T=T_{NM},\text{ or }T=T_D,\\
\Min, & T=T_G,\\
\Prod, & T=T_P.
\end{cases}
\]
It factors through the quotient category $\CTNorm_{\Ucal}$.
\end{proposition}

\begin{proof}
The proof is identical to that of the previous proposition, using Proposition~\ref{prop:low-classes} instead of Proposition~\ref{prop:pairwise}.
\end{proof}

\section{Concluding remarks}

We have studied five classical $t$-norms through ultrafilter equivalence. The near-$1$ regime yields four distinct equivalence classes, while the low-value regime collapses three of the operations to the zero type. These computations give a finite quotient category and provide a concise asymptotic skeleton for the five examples.

Goldbring's general framework for ultrafilters clarifies the role of large sets, ultrafilter quantifiers, induced ultrafilters, pushforward ultrafilters, ultrapowers, and saturation in our construction \cite{Goldbring}. In particular, the ultrapower interpretation shows that the two endpoint regimes become exact algebraic phenomena on infinitesimal monads, while saturation gives a compactness principle for countable asymptotic requirements.

The relationship with fuzzy implications is also natural. Since many important fuzzy implications are generated from $t$-norms, especially through residuation, the ultrafilter quotient of $t$-norms may induce a corresponding asymptotic quotient of residual implications. Proposition~\ref{prop:residual-consequence} gives a first elementary example: the residual implications associated with the G\"odel and nilpotent minimum $t$-norms become indistinguishable in the near-$1$ regime. This connects the present finite asymptotic skeleton with the algebraic and analytical theory of fuzzy implications developed in \cite{BaczynskiJayaram}.

The results are intentionally limited in scope. They do not constitute a classification of all continuous $t$-norms. A natural next problem is to determine whether analogous finite asymptotic skeletons can be extracted from larger classes of $t$-norms, such as ordinal sums of continuous Archimedean components or broader families of left-continuous $t$-norms. Another direction is to study whether similar ultrafilter quotients exist for residual, $(S,N)$-, QL-, and other families of fuzzy implications.

\end{document}